\documentclass[11pt]{article}
\usepackage{amsmath,amsthm, url, amsfonts}
\usepackage{amssymb}
\usepackage[left=2.54cm,top=2.54cm,right=2.54cm,bottom=2.54cm]{geometry}
\usepackage{epsfig}
\usepackage{color}                    % For creating coloured text and background

\newtheorem{theorem}{Theorem}[section]
\newtheorem{lemma}[theorem]{Lemma}
\newtheorem{claim}[theorem]{Claim}

\newtheorem{problem}[theorem]{Problem}

\renewcommand{\d}{\delta}

\newcommand\eps{\varepsilon}

\newcommand{\Prob}{\mathbb{P}}

\newcommand{\G}{{\mathcal{G}}}

\newcommand{\of}[1]{\left( #1 \right)}
\newcommand{\set}[1]{\left\{ #1 \right\}}

\newcommand{\sqbs}[1]{\left[ #1 \right]}

\renewcommand{\Pr}[1]{\mathbb{P}\sqbs{#1}}

\newcommand{\bfrac}[2]{\of{\frac{#1}{#2}}}

\newcommand{\Bin}{\textrm{Bin}}

\newcommand{\Aa}{{\mathcal{A}}}

\newcommand{\wh}[1]{\widehat{#1}}

\newcommand{\gout}[1]{\G_{n,#1\textrm{-out}}}
\newcommand{\dout}[1]{\mathcal{D}_{n,#1\textrm{-out}}}

\title{A Ramsey Property of Random Regular and $k$-out Graphs}

\author{{\Large Michael Anastos} \thanks{\texttt{manastos@andrew.cmu.edu}} \\ Department of Mathematical Sciences \\ Carnegie Mellon University \\ Pittsburgh, PA; USA
 \and {\Large Deepak Bal} \thanks{\texttt{deepak.bal@montclair.edu}} \\ Department of Mathematical Sciences \\ Montclair State University \\ Montclair, NJ; USA}
\date{}
\begin{document}
\maketitle

\begin{abstract}
In this note we consider a Ramsey property of random $d$-regular graphs, $\G(n,d)$. Let $r\ge 2$ be fixed. Then w.h.p.\ the edges of $\G(n, 2r)$ can be colored such that every monochromatic component has size $o(n)$. On the other hand, there exists a constant $\gamma > 0$ such that  w.h.p., every $r$-coloring of the edges of $\G(n, 2r+1)$ must contain a monochromatic cycle of length at least $\gamma n$.
We prove an analogous result for random $k$-out graphs.
\end{abstract}

\section{Introduction}

We are concerned with the following Ramsey-type question: if the edges of a graph are $r$-colored (not necessarily properly), what is the largest monochromatic component (or path, or cycle) which must appear? 

%Given a graph $G$ and a positive integer $r$, we define $mc_r(G)$ to be the largest integer $m$ such that every $r$-coloring of the edges of $G$ contains a monochromatic component on at least $m$ vertices.

%For $r\ge 2$, Gyarf\'{a}s [Gya] proved that if the edges of $K_n$ are colored with $r$ colors, then the largest monochromatic component is of size at least $\frac{n}{r-1}$. F\"{u}redi [Fur] proved that there is always a monochromatic component of size at least $\frac{n}{(r-1)\alpha(G)}$ for any graph $G$ with independence number $\alpha(G)$ and this is tight whenever $r-1$ is a prime power. 

This question was considered for Erd\H{o}s-R\'{e}nyi random graphs $G(n,p)$ independently by Bohman, Frieze, Krivelevich, Loh and Sudakov \cite{BFKLS} and  by Sp\"{o}hel, Steger and Thomas \cite{SST}. They proved that for every $r\ge 2$ there is a constant $\psi_r$ such that if $c < \psi_r - \eps$, then w.h.p.\ $G(n, c/n)$ admits an $r$-edge coloring where all monochromatic components are of order $o(n)$. If $c > \psi_r + \eps$, they prove that w.h.p.\ every edge coloring contains a monochromatic component of order $\Omega(n)$. The constant $\psi_r$ actually arises from the so-called $r$-orientability threshold which was discovered independently by Cain, Sanders and Wormald \cite{CSW} and by Fernholz and Ramachandran \cite{FR}. 
A graph is called \emph{$r$-orientable} if its edges can be oriented such that the maximum in-degree of any vertex is at most $r$. 
%The authors of [BFKLS] and [SST] make use of the orientability below the threshold to produce a `fragile' coloring which can be broken into smaller pieces with the removal of a small number of edges. Above the  
Recently, Krivelevich \cite{K17} improved upon the results in \cite{BFKLS} and \cite{SST} by showing that w.h.p.\ every edge $r$-coloring of $G(n, \frac{\psi_r+\eps}{n})$ contains not only a linear sized monochromatic component, but actually a linear length monochromatic cycle. The result follows from a nice theorem which proves the existence of long cycles in locally sparse graphs (stated as Theorem \ref{thm:k17} below).

For regular graphs, Thomassen \cite{Tho} proved that every $3$-regular graph has a 2-coloring of its edges such that every monochromatic component is a path of length at most 5. Alon et.\ al.\ \cite{ADO} proved that every $(2r-1)$-regular graph can be edge $r$-colored such that each monochromatic component contains at most $120r - 123$ edges.
On the other hand, they prove that there exist $2r$-regular graphs on $n$ vertices such that every edge $r$-coloring contains a monochromatic cycle of length at least $\Omega(\log n)$.

Our first theorem provides an analog of the results of \cite{BFKLS}, \cite{SST} and \cite{K17} in the setting of random $d$-regular graphs, $\G(n,d)$. 
%The \emph{random regular graph} $\G(n,d)$ is a graph chosen uniformly at random from among all labeled $d$-regular graphs on vertex set $[n]$.
\newpage 
\begin{theorem}
\label{thm:main}
For each fixed $r\geq 2$, there exists a constant $\gamma>0$ such that w.h.p. 
\begin{enumerate}
\item there exists an $r$-coloring of the edges of $\G(n, 2r)$ such that the largest monochromatic component has order $o(n)$;
\item every $r$-coloring of the edges of $\G(n, 2r+1)$ contains a monochromatic cycle of length at least $\gamma n.$
\end{enumerate}
\end{theorem}

We note that as in the case of binomial random graphs, this ``threshold" also corresponds to the orientability threshold for regular graphs. Indeed, a result of Hakimi \cite{H65} says that a graph is $r$-orientable if and only if every subgraph has average degree at most $2r$. Thus $2r$-regular graphs are $r$-orientable, but $(2r+1)$-regular graphs are not. 
Our next theorem provides an analogous result for the model $\gout{k}$ where each vertex chooses $k$ random neighbors (see below for the formal definition). Again, this corresponds with the $r$-orientability ``threshold.''
\begin{theorem}
\label{thm:main-out}
For each fixed $r\ge 2$, there exists a constant $\gamma>0$ such that w.h.p. 
\begin{enumerate}
\item there exists an $r$-coloring of the edges of $\G_{n, r\textrm{-out}}$ such that the largest monochromatic component has order $o(n)$;
\item every $r$-coloring of the edges of $\G_{n, (r+1)\textrm{-out}}$ contains a monochromatic cycle of length at least $\gamma n.$
\end{enumerate}
\end{theorem}

In Sections \ref{sec:reg} and \ref{sec:out} we prove statement (i) of Theorems \ref{thm:main} and \ref{thm:main-out} respectively. In Section \ref{sec:lb}, we prove statement (ii) of Theorems \ref{thm:main} and \ref{thm:main-out}. We conclude in Section \ref{sec:con} with a discussion of a few open problems.

\subsection*{Definitions and Notation} See \cite{JLR} or \cite{Wor99} for details on random regular graphs. We use $\G(n,d)$ to refer to a graph drawn uniformly at random from all $r$-regular graphs on vertex set $[n]= \set{1,\ldots, n}$.  Further, we refer to two related models: $\G^*(n, d)$ and $\G'(n,d)$. In the \emph{configuration} or \emph{pairing model}, $\G^*(n,d)$, a set of $nd$ many \emph{configuration points} (assuming $nd$ is even) is partitioned into $n$ \emph{cells} of size $d$, each cell corresponding to a vertex of $[n]$. A perfect matching is placed on the set of configuration points and then each cell is contracted to a vertex resulting in $d$-regular multi-graph in which loops and multi-edges may appear. $\G'(n, d)$ is $\G^*(n,d)$ conditioned to have no loops. Any property which holds w.h.p.\ in $\G^*(n,d)$ or $\G'(n,d)$ also holds w.h.p.\ in $\G(n,d)$ (see Theorem 9.9 in \cite{JLR}).

See \cite{FK} or \cite{Bol} for details on random $k$-out graphs. For $1\le k \le n-1$, let 
$\dout{k}$ represent a random digraph on vertex set $[n]$ where each vertex independently chooses a set of $k$ out-neighbors uniformly at random from all $\binom{n-1}{k}$ choices.  $\gout{k}$ is a random (multi)graph obtained from $\dout{k}$ by ignoring the orientation of the arcs. 

Let $G$ be a (multi)(di)graph. We write $G=G_1+\cdots +G_\ell$ if i) all of $G,G_1,...,G_\ell$ are defined on the same vertex set, 
ii) each of $G_1,...,G_\ell$  is chosen independently and uniformly at random from a given set of graphs,
iii) $E=E(G)=E(G_1)\cup E(G_2) \cup \cdots \cup E(G_\ell)$. 

We omit floors and ceilings in certain places for ease of presentation. 

\section{Proof of Theorem \ref{thm:main} (i)}\label{sec:reg}

It is well known (see Theorem 9.43 of \cite{JLR}) that if we consider $G = H_1+H_2+\cdots + H_r$ where the $H_i$ are chosen from the set of all Hamilton cycles on vertex set $[n]$, then $G$ and $\G'(n, 2r)$ are mutually contiguous and so any property which holds w.h.p.\ in $G$ also holds w.h.p.\ in $\G(n, 2r)$. Thus the following theorem implies Theorem \ref{thm:main} (i).

\begin{theorem}
Let $r\geq 2$ be fixed and $G=H_1+H_2+\cdots +H_{r}$ be a (multi)graph on $[n]$ %$V=\{v_1,...,v_n\}$ 
 where $H_1,...,H_r$ are chosen from the set of all Hamilton cycles on $[n]$. Then, w.h.p.\@ the  edges of $G$ can be $[r]$-colored such that for every $i\in [r]$ the largest component of the graph spanned by the edges of color $i$ has order at most $O(n^{0.7})$.
\end{theorem}
\begin{proof}

We reveal $H_1$ and we relabel our vertices such that $E(H_1)=\{\set{v_{i},v_{i+1}}:i\in [n]\}$ (we identify $v_1$ with $v_{n+1}$). For $i\in [n^{0.3}]$ set 
$$V_i=\big\{v_j:\lfloor(i-1)n^{0.7}\rfloor+1\leq j \leq \lfloor in^{0.7}\rfloor\big\}.$$
 Furthermore we define the edge sets $E^*=\big\{  \set{v_{\lfloor in^{0.7}\rfloor},v_{\lfloor in^{0.7}\rfloor+1} }:i\in[n^{0.3}]\big\},$ $E_1=\{\set{u,v}\in E: u,v\in V_i \text{ for some } i\in [n^{0.3}] \},$ and for $i\in[2,r]$ we set 
$E_i=E(H_i)\setminus E_1$.
\vspace{3mm}
\\We now implement the following coloring: for $i\in [r]$ we color the edges in $E_i$ by color $i$. Additionally we color the edges in $E^*$ by color 2.
\begin{claim}\label{shortpaths}
With probability $1- o(1)$ there does not exist $2\leq i\leq r$ such that $E_i$ spans a path of length larger than $n^{0.4}$.
\end{claim}
\begin{proof}
Fix $2\leq i\leq r$ and $v\in V.$ Furthermore let $\sigma = \sigma_i$ be one of the two permutations associated with $H_i$. We consider  the exploration of the path $v,\sigma(v),\sigma^2(v),\ldots$ induced by $H_i$ executed as follows.
For $t\in \mathbb{N}$ given the path $v,\sigma(v),\cdots, \sigma^{t-1}(v)$ 
we query  $\sigma^{t}(v)$. 
We define the stopping time of the exploration to be $\tau=\min\big\{t:\set{\sigma^{t-1}(v), \sigma^{t}(v)}\in E_1\big\}$.
\vspace{3mm}
\\ Now let $1\le t\leq n^{0.4}$ and let $z_t\in [n^{0.3}] $ be such that $\sigma^{t-1}(v)\in V_{z_t}$. Furthermore assume that we have 
not stopped the exploration of the path at time $t-1$ i.e.\@ $\tau>t-1$. 
Then  $\sigma^{t}(v)$ is uniformly distributed in $V\setminus\{v,\sigma(v),\ldots, \sigma^{t-1}(v)\}$. Thus 
 $\sigma^{t}(v)\in V_{z_t}$ with probability at least $\frac{|V_{z_t}|-t}{n-t}\geq \frac{0.9n^{0.7}}{n}.$ Therefore
$$ \Prob(\tau> n^{0.4})=\prod_{j=1}^{n^{0.4}}\Prob(\tau>j | \tau>j-1) 
                     \leq \prod_{j=1}^{n^{0.4}} \bigg(1-\frac{0.9}{n^{0.3}}\bigg)  \leq e^{-n^{0.4}\cdot\frac{0.9}{n^{0.3}}}=e^{-0.9n^{0.1}}.$$
Thus the probability that the subpath of $v,\sigma(v),\sigma^2(v),\ldots$  that is incident to vertex $v$ and is induced by  the edges in $E_i$ is larger than $n^{0.4}$ is less than                          
$e^{-0.9n^{0.1}} = o(1/n^2)$. Taking a union bound over all $2\le i\le r$ and $v\in V$, we find that every segment of length $n^{0.4}$ of each $H_i$ contains an edge of $E_1$. Thus each $E_i, i\ge 2$ consists of disjoint paths of length less than $n^{0.4}$
 %\deepak{Do we really need the paths going the other way? Can't we union bound here and say that every segment of each $H_i$ of length at most $n^{.4}$ contains an edge of $E_1$ and so each $E_i, i\ge 2$ induces only paths of length $\le n^{.4}$?}
%By symmetry we have that  the subpath of $v_j,\sigma_i^{-1}(v_j),\sigma_i^{-2}(v_j),\ldots$  that is incident to vertex $v_j$ and is induced by  the edges in $E_i$ is larger than $n^{0.4}$ is less than                          $e^{-0.9n^{0.1}}$. Therefore with probability at least $1-\frac{1}{n^2}$ the path that is induced by $E_i$ and spans $v_j$ is of length at most  $2n^{0.4}$. By taking union bound over $i,j$ the claim follows. 
\end{proof}
Observe that the largest component of the graph spanned by $E_1$ is spanned by some $V_j$, 
$j\in[n^{0.3}]$ and therefore it has size $\lceil n^{0.7} \rceil$. For 
$2\leq i\leq r$,  $E_i$ is the union of vertex disjoint paths each of which w.h.p.\@ has length at most $n^{0.4}$ (see Claim \ref{shortpaths}). Thus the largest component spanned by color $i$, for  $3\leq i\leq r$ is of size at most $n^{0.4}$. 
Finally we colored by color 2 the edges in $E_2\cup E^*$. Hence, since $|E^*|=n^{0.3}$ and any component spanned 
by $E_2$ has size at most $n^{0.4}$ we have that the largest component spanned by the edges of color $2$ has w.h.p.\@ size at most
$(n^{0.3}+1)  \cdot n^{0.4}.$

\end{proof}

\section{Proof of Theorem \ref{thm:main-out} (i)} \label{sec:out}
Let $D$ be the directed graph $D=D_1+D_2+...+D_k$ where $D_i$ is chosen from all directed graphs on $V$ where every $v\in V$ has out-degree 1 (we forbid loops) and $|V|=n$. 
%So each $D_i$ is the functional digraph of a random mapping with the restriction that loops are not allowed (see [FrKa] or [Bol] for information about random mappings). 
 Then with probability bounded away from zero, every vertex in $D$ has $k$ distinct neighbors and we obtain $\dout{k}$. Let $G$ be the (multi)graph obtained by ignoring the orientation of the arcs in $D$. Then we have that $\gout{k}$ is contiguous with respect to $G$, i.e. every statement which holds w.h.p.\ in $G$ also holds w.h.p.\ in $\gout{k}$. Thus the following theorem implies Theorem \ref{thm:main-out} (i). 
\begin{theorem}
Let $r\ge 2$ be fixed and let $D=D_1+D_2+...+D_r$ be a (multi)digraph where $D_i$ is chosen from all directed graphs on $[n]$ where every $v\in[n]$ has out-degree 1. Let $G$ be the (multi)graph obtained by ignoring the orientations of the arcs in $D$. Then, w.h.p.\@ the  edges of $G$ can be $[r]$-colored such that for
every $i\in [r]$ the largest component of the graph spanned by the edges of color $i$ has order at most $O(n^{0.9})$.
\end{theorem}
\begin{proof}

%\vspace{3mm}
We start 
by constructing  a partition of $V$ into $(1+o(1))n^{0.1}$ sets each of size $(1+o(1))n^{0.9}$ and finding a set $E^*\subset E(D_1)$ of size at most $n^{0.2}$ such that if an arc in $E(D_1)$ has its endpoints in different sets of the partition then it belongs to $E^*$. 
\vspace{3mm}
\\We construct the partition and $E^*$ as follows. Remove one arc from each cycle of $D_1$, and add it to $E^*$.
%and then we remove it from $D_1$.
Observe that in expectation $D_1$ has $O(\log n)$ cycles (see e.g. Section 14.5 of \cite{Bol}) hence w.h.p.\ we have added at most $n^{0.1}$ arcs to $E^*$. At the same time the removal of those arcs turns $D_1$ into the union of vertex disjoint  in-arborescences. Henceforward we implement the following algorithm
\vspace{3mm}
\\While $D_1$ contains an in-arborescence of order larger than $n^{0.85}$:
\begin{itemize}
\item Pick a vertex $v$ such that in $D_1$, $v$ is reachable by at least $n^{0.85}$ vertices but none of $v$'s in-neighbors have this property.
\item  Remove every in-arc of $v$ from $D_1$ and add it to $E^*$.
\end{itemize}
%{\color{red}
%While $D_1$ contains an in-arborescence of order larger than $n^{0.85}$:
%\begin{itemize}
%\item  Choose a leaf $v$ of an in-arborescence of order larger than $n^{0.85}$.
%\item  Let $w_v$ be the closest ancestor\deepak{descendant?} of $v$ which is reachable by at least $n^{0.85}$ vertices in $D_1$.
%\item  Remove every in-arc of $w_v$ from $D_1$ and add it to $E^*$.
%\end{itemize}}
The maximum in-degree of $D_1$ is w.h.p.\@ less than $\log n$. Therefore at every iteration we add to $E^*$ at most $\log n$ arcs. Moreover after each iteration at least $n^{0.85}$ additional vertices are spanned by in-arborescences of size at most $n^{0.85}$. Therefore there are at most $n^{0.15}$ iterations and w.h.p.\@ $|E^*|\leq n^{0.15}\log n+n^{0.1}<n^{0.2}.$
\vspace{3mm}
\\The removal of $E^*$ breaks $D_1$ into in-arborescences $A_1,...,A_\ell$ each of size at most $n^{0.85}$. For $1\leq j< n^{0.1}$ define $h_j=\min\{i\in[\ell]: |\cup_{d\leq i} V(A_d)|\geq j n^{0.9}\}$. Also set $h_0=0$ and $h_{n^{0.1}}=\ell$.
 Partition $V$ into $V_1,...,V_{n^{0.1}}$, where $V_j= \cup_{h_{j-1}< b \leq{h_j}}V(A_b).$ Finally define $E_1=\{(u,v)\in E(D): u,v\in V_j \text{ for some } j\in [n^{0.1}]\},$ and for $i\in[2,r]$ we set
$E_i=E(D_i)\setminus E_1$.
\vspace{3mm}
\\We now implement the following coloring of the edges of $G$: for $i\in [r]$ we color the edges of $G_k$ obtained from $E_i$ by color $i$. Additionally we color the edges obtained from $E^*$ by color 2.
\begin{claim}\label{shortpaths-out}
With probability $1- o(1)$ there does not exist $2\leq i\leq r$ such that $E_i$ spans a path of length larger than $n^{0.15}$.
\end{claim}
\begin{proof}
Fix $2\leq i\leq r$ and $v\in V.$ Furthermore let $f = f_i$ be a function $f:V \to V$ such that 
$E(D_i)=\set{(v,f(v)):v\in V }$. We consider  the exploration of the walk $v,f(v),f^2(v),\ldots$ induced by $E_i$ executed as follows.
For $t\in \mathbb{N}$ given the walk $v,f(v),\cdots, f^{t-1}(v)$ 
we query  $f^{t}(v)$. 
We define the stopping time of the exploration to be \[\tau=\min\big\{t\,:\,(f^{t-1}(v), f^{t}(v))\in E_1 \quad\textrm{ or } \quad f^t(v) = f^i(v) \textrm{ for some $i<t$} \big\}.   \]
%\deepak{This seems to be an annoying technicality, but really the "path" explored is a "walk" since $f^t(v)$ may land on $f^i(v)$ for some $i<t$. So I guess our stopping time should also include that this happens.}
\\ Now let $1\le t\leq n^{0.15}$ and let $z_t\in [n^{0.1}] $ be such that $f^{t-1}(v)\in V_{z_t}$. Furthermore assume that we have 
not stopped the exploration of the walk at time $t-1$ i.e.\@ $\tau>t-1$. 
Then  $f^{t}(v)$ has not yet been exposed and is uniformly distributed in $V\setminus\{ f^{t-1}(v)\}$. Thus 
 $f^{t}(v)\in V_{z_t}$ with probability $\frac{|V_{z_t}|-1}{n-1}\geq \frac{0.9n^{0.9}}{n}.$ Therefore
$$ \Prob(\tau> n^{0.15})=\prod_{j=1}^{n^{0.15}}\Prob(\tau>j | \tau>j-1) 
                     \leq \prod_{j=1}^{n^{0.15}} \bigg(1-\frac{0.9}{n^{0.1}}\bigg)  \leq e^{-n^{0.15}\cdot\frac{0.9}{n^{0.1}}}\leq \frac{1}{n^2}.$$
Thus the probability that the sub-walk of $v,f(v),f^2(v),\ldots$ incident to $v$ and induced by $E_i$    has length larger than $n^{0.15}$ is less than                      
$1/n^2$. Taking a union bound over all $2\le i\le k$ and $v\in V$ the claim follows.
\end{proof}
\begin{claim}\label{size}
With probability $1- o(1)$ there does not exist $2\leq i\leq r$ such that $E(D_i)$ spans an in-arborescence of height at most $n^{0.15}$ and order greater than $n^{0.7}$.
\end{claim}
\begin{proof}
Let $2\le i \le r$, $v\in V$. We explore the in-arborescence rooted at $v$ in $D_i$ using breadth-first search. 
%\red{Let $L_0=\{v\}$, $L_j$ be the vertices at level $j$, $|L_j|=\ell_j$ and $t_j = \sum_{0\le i\le j}\ell_i$.}
 For $j\ge 0$ let $\ell_j$ be the number of vertices at level $j$ (where level 0 contains only $v$). Furthermore let $t_j = \sum_{0\le i\le j}\ell_i$.
 Then, given $\ell_0,...,\ell_h$ we have that $\ell_{h+1}$ is distributed as $\Bin\of{n-t_h,\frac{\ell_h}{n-t_{h-1}}}$ i.e a binomial random variable with $n-t_{h}$ trials and probability of success $\ell_h/(n-t_{h-1})$.
 Hence $\ell_{h+1}$ is dominated by  
$\Bin(n_h,\ell_h/n_h)$ where $n_h = n - t_{h-1}$.
 Observe that using the Chernoff bound, 
 i.e. $\Prob\big[\Bin(n,p)\geq (1+\epsilon)np   \big]\leq \exp\{-\epsilon^2np/3\}$ (see e.g. \cite{JLR}), for any $h\in [n]$
  we have 
 $$\Prob\big[\ell_{h+1}>2n^{0.51}\,|\, \ell_h<n^{0.51}\big]\leq \Pr{\Bin\of{n_h,\frac{n^{0.51}}{n_h}}>2n^{0.51}}
\leq\exp \{-n^{0.51}/3\}. $$
Furthermore 
\begin{align*}
\Prob\big[\ell_{h+1}>\big(1+{n^{-0.25}}\big)\ell_h \,\vert\, \ell_h>n^{0.51}\big]&\leq 
 \Prob\bigg[\ell_{h+1}>\bigg(1+\frac{n^{0.05}}{\ell^{0.5}_h }\bigg)\ell_h  \,\bigg\vert\, \ell_h>n^{0.51}\bigg]\\
 &\leq \Prob\bigg[\Bin(n_h,\ell_h/n_h)>  \bigg(1+\frac{n^{0.05}}{\ell^{0.5}_h }\bigg)\ell_h \bigg]
\leq\exp\{-n^{0.1}/3 \}.
\end{align*}
Therefore with probability at least $1-o(1/n)$ for every $h\in [n^{0.15}]$  we have that $\ell_h\leq (1+n^{-0.25})^h 2n^{0.51}\leq 4n^{0.51}$. 
Thus with probability at least $1-o(1/n)$ in $D_i$ any in-arborescence of height at most $n^{0.15}$ rooted at $v$ spans less than $n^{0.15}\cdot 4n^{0.51}\leq
n^{0.7}$ vertices.
 By taking a union bound over all $2\le i\le r$ and $v\in V$ the claim follows.
\end{proof}
Now observe that the largest component of the graph spanned by the edges obtained from $E_1$ is contained in
%spanned by 
some $V_j$, $j\in[n^{0.1}]$ and therefore the largest component of color 1 has size at most $(1+o(1))n^{0.9}$.
Now consider $E_i$ for $2\le i \le r$. Each component of the digraph induced by $E_i$ is either an in-arborescence or unicyclic in which case we can view it as an in-arborescence plus an edge. By Claim \ref{shortpaths-out}, each such component (viewed as an in-arborescence) has height at most $n^{0.15}$ and thus by Claim \ref{size} has order at most $n^{0.7}$.
%For $2\leq i\leq r$,  $E_i$ is the union of the edge sets of in-arborescences each having height at most $n^{0.15}$ (see Claim \ref{shortpaths-out}). Each such arborescence spans at most $n^{0.7}$ vertices (see Claim \ref{size}). 
Thus the largest component spanned  by color $i$, $3\leq i\leq r$, is of order at most $n^{0.7}$.
Finally the largest component spanned by edges obtained from $E_2$ is of size $n^{0.7}$ and we have
$|E^*|\leq n^{0.2}$. 
Therefore the largest component spanned by color 2 is of order $O(n^{0.9})$.
\end{proof}

\section{Proof of Theorem \ref{thm:main} (ii) and Theorem \ref{thm:main-out} (ii)} \label{sec:lb}

This section follows closely the recent paper of Krivelevich \cite{K17}. We begin by stating a Theorem from \cite{K17} which we will use. This theorem says that graphs with decent global density, but relatively smaller local density must contain a long cycle.
\begin{theorem}[Theorem 2 of \cite{K17}]\label{thm:k17}
Suppose reals $c_1 > c_2 >1$ and a positive integer $k$ satisfy $(\frac{k}{2} -1)\of{\bfrac{c_1}{c_2}^{1/2} -1} \ge 2$. Let $G=(V,E)$ be a graph on at least $k$ vertices satisfying
\begin{enumerate}
\item $|E(G)| \ge c_1 |V|$

\item every subset $S\subseteq V$ of size $|S| \le k$ satisfies $e(S) \le c_2|S|$.

\end{enumerate}
Then $G$ contains a cycle of length at least $(\frac{k}{2} -1)\of{\bfrac{c_1}{c_2}^{1/2} -1}.$
\end{theorem}

The following lemma (whose proof is almost identical to that of Proposition 3 of \cite{K17}) verifies that random regular graphs satisfy the ``local sparseness'' condition (ii) of Theorem \ref{thm:k17}

\begin{lemma}\label{lem:localsparse}
Let $d\ge 2$ be an integer, let $d>c>1$ and let  $\delta =\of{\frac{1}{3}\cdot \frac{c^c}{e^{1+c}d^c}}^{\frac{1}{c-1}}$. Then, w.h.p. \@ every subset $S$ of $\G(n,d)$ of size $|S|\le \delta n$ satisfies $e(S) \le c|S|$.
\end{lemma}
\begin{proof}
%For this proof we use the standard configuration or pairing model for random regular graphs (see [JLR] or [Wor] for a details version of what follows). In this model, $\G^*(n,d)$, a set of $nd$ \emph{configuration points} (assuming $nd$ is even) is partitioned into $n$ \emph{cells} of size $d$. A perfect matching is placed on the set of configuration points and then each cell is contracted to a vertex resulting in $d$ regular multi-graph in which loops and multi-edges may appear. Any property which holds w.h.p.~in $\G^*(n,d)$ also holds w.h.p.~in $\G(n,d)$.

The probability that there is a subset $S$ in $\G^*(n,d)$ with $|S|\le \d n$ and $e(S) > ck$ is at most 
\begin{align*}
\sum_{k\le \d n}\binom{n}{k}\binom{dk}{ck}\bfrac{dk}{dn}^{ck} &\le \sum_{k\le \d n}\bfrac{ne}{k}^k\bfrac{dke}{ck}^{ck}\bfrac{k}{n}^{ck}
\le \sum_{k\le \d n}\sqbs{\frac{e^{1+c}d^c}{c^c} \bfrac{k}{n}^{c-1} }^k
\end{align*}

To get the first expression, we choose $k$ cells corresponding to $S$, then we choose $ck$ configuration points within those $dk$ configuration points. 
$\frac{dk}{dn}$ is a bound on the probability that one of these $ck$ points matches to one of the $dk$ points corresponding to $S$. Let $u_k = \sqbs{\frac{e^{1+c}d^c}{c^c} \bfrac{k}{n}^{c-1} }^k$. If $k\le \ln n$, then 
$u_k \le\sqbs{O(1) \bfrac{\ln n}{n}^{c-1} }^k $, so  $\sum_{k\le \ln n}u_k = o(1)$.
If $\ln n < k \le \d n$, then using the value of $\d$, we get
\[u_k \le \sqbs{\frac{e^{1+c}d^c}{c^c} \d^{c-1} }^k =\frac{1}{3^k} = o(1/n)\] and so w.h.p. there is no subset violating the property.

\end{proof}

\begin{proof}[Proof of Theorem \ref{thm:main} (ii)]

Let $G\sim \G(n, 2r+1)$. Then Lemma \ref{lem:localsparse} applied to $G$ with $d=2r+1$ and $c = 1 + \frac{1}{4r}$ implies that every subset $S$ of size $|S|\le \d n$ has $e(S)\le \of{1 + \frac{1}{4r}}|S|$. Note that this property is inherited by any subgraph of $G$.

Let the edges of $G$ be $r$-colored and let $\wh{G}$ be the subgraph whose edges are the majority color. Then $|E(\wh{G})| \ge \frac{1}{r}|E(G)| = \frac{1}{r}\frac{2r+1}{2}n = \of{1 + \frac{1}{2r}}n$. Thus Theorem \ref{thm:k17} applied to $\wh{G}$ with $c_1 = 1 + \frac{1}{2r}$,  $c_2 = 1 + \frac{1}{4r}$ and $k = \d n$ implies that $\wh{G}$ has a cycle (and thus $G$ has a monochromatic cycle) of length at least

\[\of{\frac{\d n}{2} -1}\of{\bfrac{1+\frac{1}{2r}}{1+\frac{1}{4r}}^{1/2} -1}\ge \gamma n\] for appropriate $\gamma>0$.

\end{proof}

\begin{proof}[Proof Sketch of Theorem \ref{thm:main-out} (ii)]
The proof that $\gout{d}$ satisfies Lemma \ref{lem:localsparse} is essentially the proof for $\G(n,d)$ verbatim. Since $|E(\gout{(r+1)})|=(r+1)n$ we may apply Theorem \ref{thm:k17} to the graph induced by the majority color with $c_1 = 1+\frac 1r$,  $c_2 = 1+\frac{1}{2r}$ and $k=\delta n$ to complete the proof as above.

\end{proof}

\section{Conclusion} \label{sec:con}
We note that much of the work related to this problem concerns vertex colorings rather than edge colorings. See \cite{ADO, BS, HST, LMST, MP}. Coloring the vertices of a graph such that each color class induces only small components is a natural relaxation of proper coloring. It would be very interesting to consider these ``bounded monochromatic component'' problems in the context of random regular graphs. As just one example, in \cite{HST} it is proved that every $4$-regular graph $G$ has a vertex partition $V = V_1 \cup V_2$ such that $G[V_1]$ and $G[V_2]$ contain only components of order at most 6. It is also noted that in general 6 cannot be replaced by a number less than 4. One can ask for the best number which can be used if restricting attention to random 4-regular graphs.

In this note, we have shown that almost every $2r$-regular graph admits an $r$-edge-coloring where every component has $O(n^{0.7})$ many vertices. Our argument can be improved to give $O(n^{2/3 + o(1)})$ but an obvious open problem is to improve this upper bound for random $2r$-regular graphs.
%Observe that we did not prove that our argument can be turned into an algorithm that runs in polynomial time.
From the algorithmic side, we ask the following question.
\begin{problem}
Does there exist a polynomial time algorithm which $r$-colors the edges of a random $2r$-regular graph such that w.h.p.\ every monochromatic component is of order $o(n)$?
\end{problem}
We note that our proof could solve this problem if one could %modify our proof. Such a proof would require
find an algorithm which decomposes a random $2r$-regular graph into $r$ Hamilton cycles such that the probability that the algorithm outputs  a ``bad'' $r$-tuple of Hamilton cycles i.e.\@  one that does not satisfy Claim  \ref{shortpaths}, is $o(1)$.

%Therefore the question of giving a polynomial time algorithm that $r$-edge colors a random $2r$-regular graph such that w.h.p.\@ 
%every monochromatic component is of order $o(n)$ is still open. 

Another extension concerns \emph{online} version of the above problem.
Let  $e_1,e_2,...,e_\tau$ be a random permutation of the edges of $\G(n, 2r)$. For $1\leq i\leq \tau$ at step $i$ the edge $e_i$ is revealed. 
The objective is to find an algorithm $\Aa$ that runs in polynomial time which, on step $i$, assigns a color from $[r]$ to $e_i$ without any knowledge of $e_{i+1},...,e_\tau$. $\Aa$ must maintain w.h.p.\ that the size of every monochromatic component is $o(n)$ until all edges have been revealed.
%Another extension concerns \emph{online} version of the above problem. That is to find an algorithm $\Aa$ that runs in polynomial time such that the following is true.
%Let $e_1,e_2,...,e_\tau$ be a random permutation of the edges of $\G(n, 2r)$. For $1\leq i\leq \tau$ at step $i$ the edge $e_i$ is revealed. Then $\Aa$, without having any knowledge of $e_{i+1},...,e_\tau$, assigns a color from $[r]$ to $e_i$.
%%colors $e_i$ with a color from $[r]$. 
%\red{The objective of $\Aa$ is to maintain that the size of every monochromatic component is $o(n)$ until all the edges have been revealed. 
%%That is  once all the edges have been revealed 
%%w.h.p.\@ every monochromatic component induced by the coloring given by $\Aa$ to be of size $o(n)$.
%}
In  \cite{BFKLS}, Bohman et.\ al.\
% Frieze, Krivelevich, Loh and Sudakov 
consider both the online and the offline version for $G(n,p)$. However the ranges of $p$ for which they proved that $G(n,p)$ can be $r$-colored such that w.h.p.\@ 
every monochromatic component is of order $o(n)$ differ in the two settings.

A final interesting problem is to determine the best bound for \emph{arbitrary} $2r$-regular graphs.  Alon et.\ al.\ \cite{ADO} proved that every $(2r-1)$-regular graph admits an $r$-edge-coloring with bounded size components whereas Theorem \ref{thm:main} (ii) shows that there exist $(2r+1)$-regular graphs such that every coloring contains a linear order component (actually cycle). 
The following problem is essentially posed in \cite{LMST} for $r=2$, but we state it here.

\begin{problem}
Given $r\ge 2$, what is the smallest integer $f_r(n)$ such that every $2r$-regular graph on $n$ vertices admits an $r$-edge-coloring where all components have order at most $f_r(n)$? Is $f_r(n)$ sublinear? 
\end{problem}
\noindent The construction in \cite{ADO} provides a lower bound of $f_r(n) = \Omega(\log n)$. Perhaps the consideration of random $2r$-regular graphs could lead to an improvement of this lower bound.

\bibliographystyle{plain}

\bibliography{random-reg-ramsey}

\end{document}